\documentclass[11pt,reqno]{amsart}

\usepackage[colorlinks,bookmarks=false,linkcolor=blue,linktocpage]{hyperref}

\newtheorem{theorem}{Theorem}
\newtheorem{lemma}[theorem]{Lemma}
\newtheorem{corollary}[theorem]{Corollary}
\newtheorem{proposition}[theorem]{Proposition}

\newtheorem*{rem*}{Remark}

\newcommand{\lex}{{\operatorname{lex}}}
\newcommand{\lm}{{\operatorname{lm}}}
\newcommand{\lc}{{\operatorname{lc}}}
\newcommand{\Mon}{\operatorname{Mon}}
\newcommand{\Spec}{\operatorname{Spec}}
\newcommand{\mm}{{\mathfrak m}}
\newcommand{\pp}{{\mathfrak p}}
\newcommand{\upto}{,\ldots,}
\newcommand{\gen}[1]{\langle{#1}\rangle}
\newcommand{\Quot}{\operatorname{Quot}}

\renewcommand{\phi}{\varphi}
\newcommand{\dimv}{\dim_{\operatorname{v}}}
\newcommand{\RR}{{\mathbb R}}
\newcommand{\QQ}{{\mathbb Q}}
\newcommand{\ZZ}{{\mathbb Z}}
\newcommand{\NN}{{\mathbb N}}
\newcommand{\A}{{\bf A}}
\newcommand{\V}{{\bf V}}
\newcommand{\R}{{\bf R}}
\newcommand{\W}{{\bf W}}
\newcommand{\B}{{\bf B}}

\begin{document}

\title{Valuative dimension and monomial orders}

\author{Gregor Kemper}
\address{Technische Universi\"at M\"unchen, Zentrum Mathematik - M11,
Boltzmannstr. 3, 85748 Garching, Germany}
\email{kemper@ma.tum.de}

\author{Ihsen Yengui}
\address{Department of Mathematics,  Faculty of Sciences of
  Sfax, University of Sfax, 3000 Sfax, Tunisia}
\email{ihsen.yengui@fss.rnu.tn}

\thanks{The second author thanks the Alexander von Humboldt Foundation
  for funding his stay at the Technische Universit\"at M\"unchen,
  during which the research for this paper was done.}


\begin{abstract}
  The main result from this note provides a constructive
  characterization of the valuative dimension, which bears a strong
  analogy to Lombardi's~\cite{Lombardi:2002} constructive
  characterization of the Krull dimension. While Lombardi's
  characterization uses the lexicographic monomial order, ours uses
  the graded (reverse) lexicographic order or, in fact, any graded
  rational monomial order. Apart from this, the paper contains some
  related results and some examples which readers may find
  illuminating.
\end{abstract}

\maketitle

\section*{Introduction}

In 2002 Lombardi~\cite{Lombardi:2002} characterized the Krull
dimension of a commutative ring $\A$ by means of certain relations
between the elements of $\A$. More precisely, he showed that for a
positive integer~$n$, we have $\dim(\A) < n$ if and only if all
elements $a_1 \upto a_n \in \A$ satisfy a relation
$P(a_1 \upto a_n) = 0$, where $P \in \A[X_1 \upto X_n]$ is a
polynomial whose lexicographically smallest monomial has
coefficient~$1$. Twelve years later, the first author and
Trung~\cite{Kemper:Trung:2014} showed that if $\A$ is Noetherian, then
this result extends to any monomial order, not just the lexicographic
one.

The research for this note started as an attempt to show that the
hypothesis on Noetherianness can be dropped from the results
in~\cite{Kemper:Trung:2014}. However, instead of a proof, we found
some counterexamples, which are presented in this paper. These
examples seemed to suggest that the valuative dimension had some role
to play, and indeed we ended up showing that any graded rational
monomial order (which we define, in a rather obvious way, in
Section~\ref{sPreliminaries}) measures the valuative dimension in
precisely the same way as the lexicographic order measures the Krull
dimension by Lombardi's result. This is the content of
Theorem~\ref{tVdim} of this paper, providing a constructive
characterization of the valuative dimension.

A different constructive characterization of the valuative dimension
was given by Coquand~\cite{Coquand:2009}. This characterization is in
terms of distributive lattices and seems to be much less elementary
than the one given in this note. Moreover, Coquand's characterization
is restricted to the case of integral domains, whereas ours does not
require this hypothesis.

Apart from the characterization of the valuative dimension, this note
contains a few other results, which can be summarized as follows,
using the language explained in Section~\ref{sPreliminaries}.

\begin{itemize}
\item For a rational monomial order~$<$, the maximal number of
  $<$-independent elements of a ring lies between its Krull dimension
  and its valuative dimension (Theorems~\ref{tDim}
  and~\ref{tVdim}\eqref{tVdimA}).
\item But if~$<$ is only a rational preorder, the maximal number of
  $<$-independent elements of a ring may be smaller than its Krull
  dimension (Proposition~\ref{pW}).
\item If~$<$ is an irrational monomial order, the maximal number of
  $<$-independent elements of a ring may be smaller than the Krull
  dimension or larger than the valuative dimension
  (Propositions~\ref{pW} and~\ref{pV}).
\item For a local ring, the maximal number of
  analytically independent elements does not exceed the valuative
  dimension (Corollary~\ref{cAnalytic}).
\end{itemize}

\noindent {\bf Acknowledgment.} We thank Henri Lombardi for fruitful
and interesting conversations.

\section{Preliminaries} \label{sPreliminaries}

In this note a {\em ring} is always understood to be a commutative
ring with unity.

We will be working with monomial orders on a polynomial ring
$\A[X_1 \upto X_n]$ over a ring $\A$. Sometimes we will also consider
monomial preorders in the sense of \cite{Kemper:Trung:2014}, which
basically are monomial orders where ties between different monomials
are allowed. According to \cite[Theorem~1.2]{Kemper:Trun:VanAnh:2017},
every monomial preorder~$<$ is given by a matrix
$M \in \RR^{m \times n}$ for some $m > 0$ in the following way:
\[
  X_1^{e_1} X_2^{e_2} \cdots X_n^{e_n} < X_1^{f_1} X_2^{f_2} \cdots
  X_n^{f_n} \qquad \Longleftrightarrow \qquad M \cdot
  \begin{pmatrix} e_1 \\ \vdots \\ e_n \end{pmatrix} <_\lex M \cdot
  \begin{pmatrix} f_1 \\ \vdots \\ f_n \end{pmatrix}
\]
($e_i,f_i \in \NN$), where $\lex$ denotes the lexicographic order with
$X_1 > X_2 > \cdots > X_n$, applied to exponent vectors in $\ZZ^n$. A
given matrix $M$ defines a monomial preorder if and only if all
columns are nonzero and their first nonzero entry is positive. Notice
that different matrices may define the same monomial preorder. For
example, adding a multiple of a row of $M$ to a lower row does not
change the preorder. By doing this repeatedly, we can achieve that $M$
has nonnegative entries, so from now on we will assume
$M \in \RR_{\ge 0}^{m \times n}$. Let us call a preorder {\em
  rational} if it can be defined by a matrix with rational entries,
which can then be assumed to be nonnegative integers. All monomial
orders used in practice are rational. A rational preorder is a
monomial order (i.e., there are no ties between monomials) if and only
if $M$ has rank~$n$, so in this case we may assume $m = n$. By
contrast, irrational preorders can be orders even if $m < n$, for
example if $M$ consists of a single row of real numbers that are
linearly independent over $\QQ$. A monomial preorder is said to be
{\em graded} if the first row of $M$ has only positive entries. It is
sometimes convenient to extend a monomial preorder to the monomials in
the Laurent polynomial ring $\A[X_1^{\pm 1} \upto X_n^{\pm
  1}]$. If~$<$ is a monomial ordering and
$P \in \A[X_1^{\pm 1} \upto X_n^{\pm 1}]$ is a nonzero (Laurent)
polynomial, we write $\lm_<(P)$ for its smallest monomial, and
$\lc_<(P)$ for the coefficient of this monomial.

Let us recall the notion of independence according
to~\cite{Kemper:Trung:2014}. Given a monomial order~$<$, a sequence
$a_1 \upto a_n \in \A$ is called {\em dependent} with respect to~$<$
(or, for short, $<$-{\em dependent}) if there exists
$P \in \A[X_1 \upto X_n]$ with $P(a_1 \upto a_n) = 0$ and
$\lc_<(P) = 1$. If~$<$ is only a preorder, then a polynomial may have
several minimal monomials. In this case it is required that among the
minimal monomials of $P$, at least one has coefficient~$1$. Let us
mention that Lombardi~\cite{Lombardi:2002} calls a sequence
pseudo-singular if it is lex-dependent, and otherwise
pseudo-regular. His result can now be stated as follows.

\begin{theorem}[Lombardi~\cite{Lombardi:2002}] \label{tLombardi}%
  Let $\A$ be a ring and~$n$ a positive integer. Then $\dim(\A) < n$
  if and only if every sequence of~$n$ elements of $\A$ is
  $\lex$-dependent.
\end{theorem}

\section{Non-Noetherian rings} \label{sNonNoetherian}

Our first example shows that Theorem~\ref{tLombardi} does not extend
to general monomial orders. We start with the monoid ring of
$\RR_{\ge 0}$ over the rational function field $\QQ(u)$. We write this
ring as $\QQ(u)\{v\}$ (not to be confused with the Puiseux polynomial
ring), and we write its elements as sums
$\sum_{\alpha \in \RR_{\ge 0}} c_\alpha v^\alpha$ with
$c_\alpha \in \QQ(u)$, where only finitely many $c_\alpha$ are
nonzero. Consider the prime ideal $\pp \subset \QQ(u)\{v\}$ of all
elements with $c_0 = 0$, and set $S := \QQ(u)\{v\} \setminus \pp$. Now
we form $\R := \QQ + S^{-1} \pp \subset \Quot\bigl(\QQ(u)\{v\}\bigr)$,
which will provide our example. The following result emphasizes the
distinguished standing that the lexicographic monomial order enjoys.

\begin{proposition} \label{pR}%
  Let~$<$ be a monomial preorder on two variables.
  \begin{enumerate}
    \renewcommand{\labelenumi}{(\theenumi)}
    \renewcommand{\theenumi}{\alph{enumi}}
  \item \label{pRa} If~$<$ is a lexicographic order, then every
    sequence of two elements of $\R$ is dependent with respect to~$<$.
  \item \label{pRb} If~$<$ is not lexicographic, there exist two
    elements of $\R$ that are independent with respect to~$<$.
  \item \label{pRc} $\R$ is a local ring of Krull dimension~$1$.
  \end{enumerate}
\end{proposition}

\begin{proof}
  \begin{enumerate}
  \item[(\ref{pRa})] Let $a = c_0 + \frac{f}{s} v^\alpha$ be an
    arbitrary element of $\R$, where $c_0 \in \QQ$,
    $f \in \QQ(u)\{v\}$, $s \in S$, and $\alpha \in \RR_{> 0}$. If
    $c_0 \ne 0$, then~$a$ is invertible since
    $a^{-1} = c_0^{-1} - \frac{c_0^{-1} f}{c_0 s + f v^\alpha}
    v^\alpha \in \R$, so~$a$ satisfies the equation
    $1 - a^{-1} X = 0$. Since~$0$ satisfies $X = 0$, we only need to
    consider two elements $a = \frac{q}{r} v^\alpha$ and
    $b = \frac{s}{t} v^\beta$ with $q,r,s,t \in S$ and
    $\alpha,\beta \in \RR_{>0}$. With
    $n := \lfloor \frac{\alpha}{\beta}\rfloor + 1$ we have
    \[
      c := \frac{b^n}{a} = \frac{r s^n}{q t^n} v^{n \beta - \alpha}
      \in \R,
    \]
    so $Y^n - c X$ provides an equation for $a,b$ whose lowest
    coefficient with respect to the lexicographic order with $X > Y$
    is~$1$. For the lexicographic order with $Y > X$, the roles of~$a$
    and~$b$ need to be interchanged.
  \item[(\ref{pRb})] \newcommand{\M}{\operatorname{Mon}}%
    Let
    $M = \left(\begin{smallmatrix} \alpha & \beta \\ \gamma &
        \delta \end{smallmatrix}\right)$ be a real matrix
    defining~$<$. The entry~$\alpha$ must be positive, since otherwise
    we could assume $(\gamma,\delta) = (1,0)$ and~$<$ would be
    lexicographic. The same argument shows $\beta > 0$. Now we set
    $a := v^\alpha$, $b = u v^\beta$ and claim that these elements of
    $\R$ are independent with respect to~$<$.

    So let $P \in \R[X,Y]$ with $P(a,b) = 0$. We need to show that no
    coefficient of~$P$ that belongs to a minimal monomial of~$P$ can
    be~$1$. For a monomial $m = X^i Y^j$ we set
    $\deg(m) = \alpha i + \beta j$. With~$\delta \in \RR_{\ge 0}$ the
    minimum degree attained by the monomials of~$P$, it follows that a
    monomial of~$P$ that is minimal needs to have
    degree~$\delta$. Let~$Q$ be the sum of all terms of~$P$ with
    degree~$\delta$. Then it suffices to show that no coefficient
    of~$Q$ is equal to~$1$. For a monomial~$m$ as above we have
    $m(a,b) = u^j v^{\deg(m)} = m(1,u) v^{\deg(m)}$. Writing
    $P = \sum_{m \in \M(P)} c_m \cdot m$ with $c_m \in \R$, we obtain
    \[
      0 = P(a,b) = v^\delta \sum_{m \in \M(P)} c_m m(1,u) v^{\deg(m) -
        \delta}.
    \]
    Since $\R$ is a domain, we may divide this equation
    by~$v^\delta$. Applying the homomorphism $\phi$: $\R \to \QQ$ that
    sends an element of $\R$ to its constant coefficient now yields
    \[
      \sum_{m \in \M(Q)} \phi(c_m) m(1,u) = 0.
    \]
    Since the monomials of $Q$ have the same degree, the expressions
    $m(1,u)$ in the above sum are pairwise distinct powers
    of~$u$. Since~$u$ is algebraically independent over $\QQ$, it
    follows that $\phi(c_m) = 0$ for every $m \in \M(Q)$. So indeed no
    coefficient of $Q$ can be equal to~$1$.
  \item[(\ref{pRc})] By Theorem~\ref{tLombardi}, part~(\ref{pRa})
    implies $\dim(\R) \le 1$. The reverse inequality can perhaps most
    quickly be seen from the chain $\{0\} \subset S^{-1}\pp$ of
    primes; it is also easy to see that~$v \in \R$ is
    $\lex$-independent. The calculation in the proof of~\eqref{pRa}
    shows that $\R \setminus S^{-1}\pp \subseteq \R^\times$, so $\R$
    is local. \qed
  \end{enumerate}
  \renewcommand{\qed}{}
\end{proof}

We now prove an easy lemma, which will be used several times. For
matrix $L = (\alpha_{i,j}) \in \ZZ^{n \times n}$, we define the
homomorphism
\[
  \phi_L\mbox{:}\ \A[X_1^{\pm 1} \upto X_n^{\pm 1}] \to \A[X_1^{\pm 1}
  \upto X_n^{\pm 1}], X_i \mapsto \prod_{j=1}^n X_j^{\alpha_{j,i}}.
\]

\begin{lemma} \label{lPrelim}%
  Let $\A$ be a ring and~$<$ be a rational monomial order on
  $\A[X_1 \upto X_n]$, given by a matrix
  $M = (\alpha_{i,j}) \in \ZZ_{\ge 0}^{n \times n}$. Then for
  $P \in \A[X_1^{\pm 1} \upto X_n^{\pm 1}]$ we have
  \[
    \lc_<(P) = \lc_\lex\bigl(\phi_M(P)\bigr).
  \]
  Moreover, let $a_1 \upto a_n \in \A$ and set
  $b_i := \prod_{j=1}^n a_j^{\alpha_{j,i}}$. If the $b_i$ form a
  $<$-dependent sequence in $\A$, then $a_1 \upto a_n \in \A$ are
  $\lex$-dependent.
\end{lemma}

\begin{proof}
  For a monomial $m = \prod_{i=1}^n X_i^{e_i}$ with $e_i \in \ZZ$ we
  have
  \[
    m < 1 \quad \Longleftrightarrow \quad \prod_{j=1}^n \prod_{i=1}^n
    X_j^{\alpha_{j,i} e_i} <_\lex 1 \quad \Longleftrightarrow \quad
    \phi_M(m) <_\lex 1.
  \]
  Since~$\phi_M$ is injective on monomials, this implies the first
  assertion. In the situation of the second assertion there exists
  $P \in \A[X_1 \upto X_n]$ with $P(b_1 \upto b_n) = 0$ and
  $\lc_<(P) = 1$. So $Q := \phi_M(P)$ satisfies $Q(a_1 \upto a_n) = 0$
  and $\lc_\lex(Q) = \lc_<(P) = 1$.
\end{proof}

As a first consequence we obtain:

\begin{theorem} \label{tDim}%
  Let $\A$ be a ring and~$<$ a rational monomial order on~$n$
  variables. If every sequence of~$n$ elements of $\A$ is
  $<$-dependent, then $\dim(A) < n$.
\end{theorem}

\begin{proof}
  The order~$<$ is given by a matrix
  $M = (\alpha_{i,j}) \in \ZZ_{\ge 0}^{n \times n}$. Let
  $a_1 \upto a_n \in \A$ and form the~$b_i$ as in
  Lemma~\ref{lPrelim}. Then the hypothesis and the lemma yield that
  $a_1 \upto a_n$ are $\lex$-dependent. From this the assertion
  follows by Theorem~\ref{tLombardi}.
\end{proof}

If $\A$ is Noetherian, then by~\cite[Theorem~3.5]{Kemper:Trung:2014},
Theorem~\ref{tDim} extends to all monomial preorders and the converse
also holds.

We now give an example of a (non-Noetherian) ring such that
Theorem~\ref{tDim} fails for all monomial preorders that are not
rational orders. It is also an example where the maximal number of
independent elements is smaller than the Krull dimension. In contrast,
the ring $\R$ constructed above has ``too many'' independent
elements. With~$u$ and~$v$ indeterminates, consider the subring
\[
  \W:= \QQ[u]_{\gen{u}} + v \, \QQ(u)[v]_{\gen{v}} \subseteq \QQ(u,v),
\]
of the rational function field, where the subscripts stand for
localization at the prime ideal generated by~$u$ and~$v$,
respectively. It is easy to see that $\W$ consists of the rational
functions with denominator not divisible by~$v$, such that the
evaluation at $v = 0$ has a denominator not divisible by~$u$.

Recall that a {\em valuation domain} is an integral domain such that
for any two elements, one divides the other. The first assertion of
the following proposition will be quite clear for readers who are
familiar with valuation domains.

\begin{proposition} \label{pW}%
  \begin{enumerate}
    \renewcommand{\labelenumi}{(\theenumi)}
    \renewcommand{\theenumi}{\alph{enumi}}
  \item \label{pWa} $\W$ is a $2$-dimensional valuation domain.
  \item \label{pWb} Let~$<$ be a monomial preorder on two variables
    that is not a rational monomial order. Then every sequence of two
    elements of $\W$ is $<$-dependent.
  \end{enumerate}
\end{proposition}

\begin{proof}
  \begin{enumerate}
  \item[\eqref{pWa}] The sequence
    \[
      \{0\} \subset v \, \QQ(u)[v]_{\gen{v}} \subset u \,
      \QQ[u]_{\gen{u}} + v \, \QQ(u)[v]_{\gen{v}}
    \]
    of prime ideals shows that $\dim(\W) \ge 2$, and the reverse
    inequality follows since $\W$ has transcendence degree~$2$ over
    $\QQ$ (see \cite[Theorem~5.5]{Kemper.Comalg}).

    A rational function in $\W$ is invertible in $\W$ if and only if
    evaluating it at $v = 0$ and then evaluating the result at $u = 0$
    yields a nonzero value. Now let $a \in \QQ(u,v) = \Quot(\W)$ be
    any nonzero rational function. There is an integer~$i$ such that
    $b := v^i a$ has numerator and denominator not divisible by~$v$,
    and there is an integer~$j$ such that $u^j b(0)$ (the evaluation
    is at $v = 0$) has numerator and denominator not divisible
    by~$u$. Therefore $v^i u^j a \in \W^\times$. This shows that the
    $v^i u^j$ form a system of representatives of
    $\QQ(u,v)^\times/\W^\times$. Moreover, we have $v^i u^j \in \W$ if
    and only if $\left(\begin{smallmatrix} i \\
        j\end{smallmatrix}\right) \ge_\lex\left(\begin{smallmatrix} 0 \\
        0\end{smallmatrix}\right)$. So for the above~$a$ we have
    $a \in \W$ or $a^{-1} \in \W$, which shows that $\W$ is a
    valuation domain.
  \item[\eqref{pWb}] The preorder~$<$ is given by a real matrix $M$
    with two columns and at most two rows. Assuming that $M$ has two
    rows, we may add a multiple of the first row to the second and
    then multiply the second row by a positive real number. This way,
    the second row may be assumed to have entries in
    $\{0,1,-1\}$. If~$<$ is irrational, it follows that the first row
    of $M$ consists of two real numbers with irrational ratio, and in
    this case the second row can be deleted since the first row
    completely determines~$<$. If, on the other hand, $<$ is not a
    monomial order, then $M$ has rank~$1$, and again the second row
    can be deleted. So in both cases we can assume
    $M = (\alpha,\beta)$ with~$\alpha,\beta \in \RR$ positive.

    For showing that all sequences of two elements $a,b \in \W$ are
    $<$-dependent, we may assume~$a$ and~$b$ to be nonzero and replace
    them by associated elements. So by the above we may assume
    $a = v^{i_1} u^{j_1}$ and $b = v^{i_2} u^{j_2}$. With
    $A := \left(\begin{smallmatrix} i_1 & i_2 \\ j_1 &
        j_2 \end{smallmatrix}\right)$, we claim that there exists a
    nonzero vector $\left(\begin{smallmatrix} e \\
        f \end{smallmatrix}\right) \in \QQ^2$ such that
    \begin{equation} \label{eqMA}%
      \alpha e + \beta f = M \cdot
      \begin{pmatrix}
        e \\ f
      \end{pmatrix} \le 0 \quad \text{but} \quad A \cdot
      \begin{pmatrix}
        e \\ f
      \end{pmatrix} \ge_\lex \begin{pmatrix} 0 \\ 0
      \end{pmatrix}.
    \end{equation}
    This is clear if the ratios of~$\alpha$ and~$\beta$ and of~$i_1$
    and~$i_2$ are different. But if the ratios are equal, then
    any~$e$, $f$ with $i_1 e + i_2 f = 0$ will satisfy the first
    inequality in~(\ref{eqMA}), so in addition we need
    $j_1 e + j_2 f \ge 0$. If $\operatorname{rank}(A) = 2$, then
    $j_1 e + j_2 f > 0$ can be achieved, and if
    $\operatorname{rank}(A) = 1$, then $j_1 e + j_2 f = 0$ is true
    whenever $i_1 e + i_2 f = 0$. Having proved the claim, we may now
    assume~$e, f \in \ZZ$ and write $e = e_1 - e_2$ and
    $f = f_1 - f_2$ with $e_i, f_i \in \NN$. Then
    $X^{e_1} Y^{f_1} \le X^{e_2} Y^{f_2}$ by the first inequality
    in~(\ref{eqMA}). Moreover,
    \[
      c := \frac{a^{e_1} b^{f_1}}{a^{e_2} b^{f_2}} = a^e b^f = v^{i_1
        e + i_2 f} u^{j_1 e + j_2 f} \in \W
    \]
    by the second inequality in~(\ref{eqMA}) and by the above
    reasoning. So the polynomial
    $P := X^{e_1} Y^{f_1} - c X^{e_2} Y^{f_2} \in \W[X,Y]$ vanishes
    at~$a,b$, and the monomial $X^{e_1} Y^{f_1}$ is minimal among the
    monomials of~$P$. This shows that~$a$ and~$b$ are
    $<$-dependent. \endproof
  \end{enumerate}
\end{proof}

\section{The valuative dimension} \label{sVdim}

Recall that the {\em valuative dimension} of a domain $\A$, denoted by
$\dimv(\A)$, is the supremum of the Krull dimensions of all overrings
of $\A$, where an overring of $\A$ is defined to be a subring of
$\Quot(\A)$ containing $\A$. It is worth mentioning that, as pointed
out by Gilmer~\cite[Theorem~30.9]{Gilmer:72}, $\dimv(\A) \le n$ iff
for any elements $t_1 \upto t_n \in \Quot(\A)$,
$\dim(\A[t_1 \upto t_n]) \le n$.
% This is the definition of the valuative dimension in the integral case
% we use in this paper.
In the case of an integral domain, this can be interpreted as a
constructive characterization of the valuative dimension, and it is in
fact the definition adopted by Lombardi and Quitt\'e in the integral
case in their book~\cite{Lombardi:Quitte:2015}.  If $\A$ is a ring
which need not be a domain, $\dimv(\A)$ is defined as the supremum of
all $\dimv(A/\pp)$ with $\pp \in \Spec(\A)$ a prime ideal (see
Jaffard~\cite[p.~56]{Jaffard:60}). It is clear that
$\dimv(\A) \ge \dim(\A)$.

In this section we prove the following theorem, which is the main
result of the paper.

\begin{theorem} \label{tVdim}%
  Let $\A$ be a ring, $n$ a positive integer, and~$<$ a rational
  monomial preorder on~$n$ variables.
  \begin{enumerate}
    \renewcommand{\labelenumi}{(\theenumi)}
    \renewcommand{\theenumi}{\alph{enumi}}
  \item \label{tVdimA} If $\dimv(\A) < n$, then every sequence of~$n$
    elements in $\A$ is dependent with respect to~$<$.
  \item \label{tVdimB} If~$<$ is a graded monomial order, then the
    converse of~\eqref{tVdimA} holds.
  \end{enumerate}
\end{theorem}

\begin{proof}
  \begin{enumerate}
  \item[(\ref{tVdimA})] Let $a_1 \upto a_n \in \A$. We start with two
    reduction steps. First, we can refine~$<$ to a rational monomial
    order by appending some rows of the unit matrix at the bottom of
    the matrix defining~$<$. If we can show that $a_1 \upto a_n$ are
    dependent with respect to the order thus obtained, then they are
    also dependent with respect to the original preorder. Therefore we
    may assume that~$<$ is a rational monomial order. Second, we
    reduce to the case that $\A$ is a domain. For this, consider the
    multiplicative set
    \[
      S := \bigl\{P(a_1 \upto a_n) | P \in \A[X_1 \upto X_n] \
      \text{with}\ \lc_<(P) = 1\bigr\} \subseteq \A
    \]
    and assume that we can show part~\eqref{tVdimA} in the domain
    case. Then for every ${\mathfrak p} \in \Spec(\A)$, the sequence
    $a_1 \upto a_n$ is $<$-dependent in $\A/\mathfrak p$, so
    $S \cap {\mathfrak p} \ne \emptyset$. This means that the
    localization $S^{-1} \A$ has no prime ideals and is therefore
    zero. So $0 \in S$, which shows the dependence of $a_1 \upto
    a_n$. Hence indeed we may assume $\A$ to be a domain. We need to
    show that $a_1 \upto a_n$ are $<$-dependent. Since this is clear
    if an~$a_i$ is zero, we may assume the~$a_i$ to be nonzero.

    By the first reduction, $<$ is given by a matrix
    $M \in \ZZ_{\ge 0}^{n \times n}$ of rank~$n$. There is a positive
    integer~$k$ such that $L := k \cdot M^{-1} \in \ZZ^{n \times
      n}$. We write $L = (\beta_{i,j})_{i,j = 1 \upto n}$ and set
    \[
      b_i = \prod_{j=1}^n a_j^{\beta_{j,i}} \in \Quot(\A) \quad (i = 1
      \upto n).
    \]
    $\B := \A[b_1 \upto b_n]$ is an overring of $\A$, so
    $\dim(\B) < n$ by hypothesis. It follows by
    Theorem~\ref{tLombardi} that the~$b_i$ are $\lex$-dependent, so
    there is a polynomial $P \in \B[X_1 \upto X_n]$ with
    $P(b_1 \upto b_n) = 0$ and $\lc_\lex(P) = 1$. Each coefficient~$c$
    of $P$ can be written as $c = C[b_1 \upto b_n]$ with
    $C \in \A[X_1 \upto X_n]$. Substituting each coefficient~$c$ by
    $C$ yields a polynomial in $\A[X_1 \upto X_n]$ that also vanishes
    at $b_1 \upto b_n$ and has lowest coefficient~$1$. So we may
    assume $P \in \A[X_1 \upto X_n]$. With the notation introduced
    before Lemma~\ref{lPrelim}, set
    $Q := \phi_L(P) \in \A[X_1^{\pm 1} \upto X_n^{\pm 1}]$. We have
    $\phi_M\bigl(\phi_L(X_i)\bigr) = X_i^k$ for all~$i$, so
    Lemma~\ref{lPrelim} shows
    \[
      \lc_<(Q) = \lc_\lex\bigl(\phi_M(Q)\bigr) = \lc_\lex\bigl(P(X_1^k
      \upto X_n^k)\bigr) = \lc_\lex(P) = 1.
    \]
    We clearly have $Q(a_1 \upto a_n) = P(b_1 \upto b_n) = 0$. By
    multiplying $Q$ with a suitable monomial, we obtain a polynomial
    in $\A[X_1 \upto X_n]$ that also vanishes at $a_1 \upto a_n$ and
    has lowest coefficient~$1$ with respect to~$<$. So $a_1 \upto a_n$
    are $<$-dependent, which finishes the proof of~\eqref{tVdimA}.
  \item[(\ref{tVdimB})] We need to show that
    $\dimv(\A/{\mathfrak p}) < n$ for every
    ${\mathfrak p} \in \Spec(\A)$. By hypothesis, every sequence
    of~$n$ elements in $\A$ is $<$-dependent, so it is also
    $<$-dependent as a sequence in $\A/\mathfrak p$. Replacing $\A$ by
    $\A/\mathfrak p$, we may therefore assume that $\A$ is a
    domain. Given an overring $\B$ of $\A$ and a sequence
    $b_1 \upto b_n \in \B$, we need to show that it is
    $\lex$-dependent; indeed, by Theorem~\ref{tLombardi}, this will
    imply $\dim(\B) < n$.

    Since $\B \subseteq \Quot(\A)$, we can choose a nonzero $a\in \A$
    such that $a b_i \in \A$ for all~$i$.  Our monomial order~$<$ is
    given by a matrix
    $M = (\alpha_{i,j}) \in \ZZ_{\ge 0}^{n \times n}$. Since it is
    graded, we can choose a positive integer~$k$ such that
    $k \alpha_{1,i} \ge \sum_{j=1}^n \alpha_{j,i}$ for all~$i$. Then
    \begin{equation} \label{eqA}%
      (a^k b_1)^{\alpha_{1,i}} \cdot \prod_{j=2}^n b_j^{\alpha_{j,i}}
      = a^{k \alpha_{1,i} - \sum_{j=1}^n \alpha_{j,i}} \cdot
      \prod_{j=1}^n (a b_j)^{\alpha_{j,i}} \in \A.
    \end{equation}
    We claim that it is enough to show that
    $a^k b_1,b_2,b_3 \upto b_n$ are $\lex$-dependent. In fact, if they
    are, then there is a polynomial $P \in \B[X_1 \upto X_n]$
    vanishing at these elements with $\lc_\lex (P) = 1$. So
    $P(a^k X_1,X_2,X_3 \upto X_n)$ vanishes at $b_1 \upto b_n$. If~$e$
    is the exponent of $X_1$ in the smallest monomial of $P$, then all
    coefficients of $P(a^k X_1,X_2 \upto X_n)$ are divisible by
    $a^{k e}$, so
    $Q := a^{-k e} P(a^k X_1,X_2 \upto X_n) \in \B[X_1 \upto
    X_n]$. Now $Q(b_1 \upto b_n) = 0$ and $\lc_\lex(Q) = 1$, which
    yields the desired $\lex$-dependence of the~$b_i$. So the claim is
    proved.

    The claim means that we may replace~$b_1$ by $a^k b_1$. Then
    by~\eqref{eqA}, $a_i := \prod_{j=1}^n b_j^{\alpha_{j,i}} \in \A$
    ($i = 1 \upto n$).  By hypothesis, $a_1 \upto a_n$ are
    $<$-dependent, so they are also $<$-dependent when considered as a
    sequence in $\B$. Now Lemma~\ref{lPrelim} shows that the sequence
    $b_1 \upto b_n$ is $\lex$-dependent, as desired. \qed
  \end{enumerate}
  \renewcommand{\qed}{}
\end{proof}

The following example shows that Theorem~\ref{tVdim}\eqref{tVdimA}
does not extend to irrational monomial preorders. Consider $\QQ\{v\}$,
the monoid ring of $\RR_{\ge 0}$ over $\QQ$, and let
$\V = \QQ\{v\}_\pp$ be the localization at the ideal~$\pp$ of all
elements with constant coefficient equal to~$0$.

\begin{proposition} \label{pV}
  \begin{enumerate}
    \renewcommand{\labelenumi}{(\theenumi)}
    \renewcommand{\theenumi}{\alph{enumi}}
  \item \label{pVa} $\V$ is a valuation domain with
    $\dim(\V) = \dimv(\V) = 1$.
  \item \label{pVb} Let~$<$ be an irrational monomial preorder on two
    variables. Then there exist two elements of $\V$ that are
    independent with respect to~$<$.
  \end{enumerate}
\end{proposition}

\begin{proof}
  \begin{enumerate}
  \item[\eqref{pVa}] It is clear that $\V$ is a valuation domain. The
    chain $\{0\} \subset \pp_\pp$ or the $\lex$-independence of~$v$
    show that $\dim(\V) \ge 1$. The reverse inequality follows by
    Theorem~\ref{tLombardi} if we can show that any two
    elements~$a,b \in \V$ are $\lex$-dependent. We may assume~$a$
    and~$b$ to be nonzero and noninvertible, and replace them by
    associate elements. This yields $a = v^\alpha$ and $b = v^\beta$
    with $\alpha,\beta \in \RR_{>0}$. With
    $n := \lceil \frac{\alpha}{\beta}\rceil$ and
    $c := v^{n \beta - \alpha} \in \V$, the relation $b^n - c a = 0$
    shows that~$a$ and~$b$ are $\lex$-dependent. So $\dim(\V) = 1$,
    and $\dimv(\V) = 1$ follows from the fact that for a valuation
    domain, the valuative and Krull dimensions coincide (see
    \cite[Chap.~IV, Prop.~1]{Jaffard:60}).
    % So $\dim(\V) = 1$, and $\dimv(\V) = 1$ follows from the well-known
    % fact that for a valuation domain, the valuative and Krull
    % dimensions coincide. We present a short proof for the convenience
    % of the reader. Let $\A$ be an overring of $\V$ and consider the
    % multiplicative subset $S := \V \cap \A^\times$ of $\V$. We need to
    % prove $\dim(\A) \le \dim(\V)$, which will follow if we can show
    % $\A = S^{-1} \V$. As $S^{-1} \V \subseteq \A$ is clear, take
    % $a \in \A$ for showing the reverse inclusion. Since
    % $\A \subseteq \Quot(\V)$ and $\V$ is a valuation domain,
    % $a \in \V$ or $a^{-1} \in \V$. In the first case,
    % $a \in S^{-1} \V$. In the second, $a^{-1} \in S$, so also
    % $a \in S^{-1} \V$.
  \item[\eqref{pVb}] We may assume that~$<$ is given by a matrix
    $M = (\alpha,\beta)$ with $\alpha,\beta \in \RR_{>0}$ linearly
    independent over $\QQ$. So with the degree of a monomial defined
    as in the proof of Proposition~\ref{pR}\eqref{pRb}, the smallest
    monomial of a polynomial is the (unique) monomial with minimal
    degree. We claim that $a := v^\alpha$ and $b := v^\beta$ are
    $<$-independent. So let
    $P = \sum_{m \in \Mon(P)} c_m \cdot m \in \V[X,Y]$ be a polynomial
    vanishing at~$a,b$. Then with~$\delta$ the degree of the smallest
    monomial of $P$ we have
    \[
      v^\delta \sum_{m \in \Mon(P)} c_m \cdot v^{\deg(m) - \delta} = 0.
    \]
    Dividing this by $v^\delta$ and applying the homomorphism $\phi$:
    $\V \to \QQ$ that sends an element of $\V$ to its constant
    coefficient yields the equation $\phi(\lc(P)) = 0$, since
    $\phi(v^{\deg(m) - \delta}) = 0$ for $m \ne \lm(P)$. So $\lc(P)
    \ne 1$, which proves the claim. \qed
  \end{enumerate}
  \renewcommand{\qed}{}
\end{proof}

The above proof also shows that if~$<$ is a monomial order on~$n$
variables given by $M = (\alpha_1 \upto \alpha_n)$ with the~$\alpha_i$
linearly independent over $\QQ$, then there are~$n$ elements of $\V$
that are $<$-independent. \\

We now present two applications of Theorem~\ref{tVdim}. The first is a
new proof of the well-known but nontrivial fact that for a Noetherian
ring, the Krull dimension and the valuative dimension coincide.

\begin{corollary}[{\cite[Chapt.~IV, Corollaire~2 to
    Th\'eor\`eme~5]{Jaffard:60}}] \label{cDim}%
  Let $\A$ be a Noetherian ring. Then $\dimv(\A) = \dim(\A)$.
\end{corollary}

\begin{proof}
  Let~$n$ be a positive integer and choose a graded rational monomial
  order~$<$ on~$n$ variables. Then by
  \cite[Theorem~2.7]{Kemper:Trung:2014}, the Krull dimension of $\A$
  is less than~$n$ if and only if every sequence of~$n$ elements is
  $<$-dependent. But by Theorem~\ref{tVdim}, this is equivalent to
  $\dimv(\A) < n$.
\end{proof}

Our second application deals with analytic independence, which is
defined, according to Matsumura~\cite{Matsumura:86}, as follows. Some
elements $a_1 \upto a_n \in \mm$ from the maximal ideal of a local
ring $\A$ are analytically independent if every homogeneous polynomial
in $\A[X_1 \upto X_n]$ vanishing at $a_1 \upto a_n$ has all its
coefficients lying in~$\mm$. To the best of our knowledge, the
following corollary is new.

\begin{corollary} \label{cAnalytic}%
  Let $\A$ be a local ring with $\dimv(\A) < n$ and let
  $a_1 \upto a_n \in \mm$ be elements from its maximal ideal. Then
  the~$a_i$ are analytically dependent.
\end{corollary}

\begin{proof}
  Applying Theorem~\ref{tVdim}\eqref{tVdimB} to the monomial preorder
  given by the matrix $M = (1,1 \upto 1)$ yields a polynomial
  $P \in \A[X_1 \upto X_n]$ vanishing at $a_1 \upto a_n$ such that the
  homogeneous part of $P$ of least degree~$d_0$ has a monomial~$t_0$
  whose coefficient is~$1$. We now turn $P$ into a homogeneous
  polynomial of degree~$d_0$ by ``partially evaluating'' it. More
  precisely, we split each monomial (of degree~$d$, say) into
  monomials of degree~$d_0$ and~$d - d_0$, and then evaluate the one
  of degree~$d - d_0$ at the~$a_i$. The resulting homogeneous
  polynomial also vanishes at the~$a_i$. The process may have changed
  the coefficient of~$t_0$, but only by adding an $\A$-linear
  combination of nonempty products of the~$a_i$. Since~$a_i \in \mm$,
  the coefficient of $t_0$ is not in~$\mm$, and corollary follows.
\end{proof}

The ring $\R$, constructed at the beginning of
Section~\ref{sNonNoetherian}, provides an example showing that in the
above result, the valuative dimension cannot be replaced by the Krull
dimension. Indeed, Proposition~\ref{pR}\eqref{pRb} shows that there
are two elements that are independent with respect to the preorder~$<$
given by $M = (1,1)$. Being independent, they must lie in the maximal
ideal of $\R$, and being $<$-independent, they are analytically
independent. Explicitly, two such elements are~$v$ and~$u v$.

On the other hand, the maximal number of analytically independent
elements may also be less than the Krull dimension. For example, if
$\A$ is a valuation domain, than any two elements are analytically
dependent (with an equation of degree~$1$); but $\A$ may have
dimension $> 1$, as exemplified by the ring $\W$ from this
note. By~\cite[Theorem~14.5]{Matsumura:86}, this cannot happen for
Noetherian rings.

\end{document}